\documentclass[11pt]{article}

\usepackage[a4paper,margin=0.91in]{geometry}
\usepackage{amsmath,amssymb,amsthm,mathtools}
\usepackage{enumitem}
\usepackage{booktabs,array,longtable}
\usepackage[hidelinks]{hyperref}
\usepackage[nameinlink,capitalize]{cleveref}
\usepackage{todonotes}

\title{Small complete 3-term progression free sets in cyclic groups and vector spaces}
\author{Bence Csajb\'ok \thanks{Department of Computer Science, 
		ELTE E\"otv\"os Lor\'and University, 
		H-1117 Budapest, P\'azm\'any P.\ stny.\ 1/C, Hungary. This paper was supported by the János Bolyai Research Scholarship of the Hungarian Academy of Sciences and partially by the National Research, Development and Innovation Fund – grant numbers ADVANCED 153080, EXCELLENCE 151504 and SNN 152582. {\texttt{bence.csajbok@ttk.elte.hu}}} \and Zolt\'an Lóránt Nagy\thanks{ Department of Computer Science, 
		ELTE  E\"otv\"os Lor\'and University, Budapest, Hungary. The author is supported by  the J\'anos Bolyai Research Grant of the Hungarian Academy of Sciences and  partially by the NRDI EXCELLENCE grant, no.  151504;  	E-mail: {\tt zoltan.lorant.nagy@ttk.elte.hu.}}}
\date{}

\newtheorem{theorem}{Theorem}[section]
\newtheorem{lemma}[theorem]{Lemma}
\newtheorem{corollary}[theorem]{Corollary}
\newtheorem{proposition}[theorem]{Proposition}
\newtheorem{remark}[theorem]{Remark}
\newtheorem{definition}[theorem]{Definition}

\crefname{theorem}{Theorem}{Theorems}
\crefname{lemma}{Lemma}{Lemmas}
\crefname{corollary}{Corollary}{Corollaries}
\crefname{proposition}{Proposition}{Propositions}
\crefname{remark}{Remark}{Remarks}
\crefname{definition}{Definition}{Definitions}
\crefname{section}{Section}{Sections}

\newcommand{\Z}{\mathbb Z}
\newcommand{\F}{\mathbb F}
\newcommand{\dotminus}{\mathbin{\dot{-}}}
\newcommand{\AP}{\mathrm{AP}}

\newcommand{\aavo}{a_{(2,-1)}}

\begin{document}
	\maketitle
	
	\begin{abstract}
		A classical extremal problem on progression free sets is to determine the maximum size of a $3$-term arithmetic progression free set in algebraic structures, for instance in intervals of integers  or in finite vector spaces.  To determine the minimum size of a complete $3$-term arithmetic progression free set is a lower-end analogue of this problem.  It is also closely related to complete caps and saturating sets in finite geometry.
		
		A simple counting argument shows that  the order of magnitude of the minimum size is at least the square root of the cardinality of the structure. Addressing two open problems, we show that this lower  bound is essentially tight.  First,  for every cyclic group $\Z_m$, we give explicit constructions of complete $3$-AP-free sets whose size is less than $2\sqrt m$.  For $m\ge81$ the constructed sets satisfy the stronger, so-called complete $(2,-1)$-avoiding property; the remaining cases $m<81$ are covered by a finite verification.  Second, we resolve the vector space variant in a weaker sense by showing that for every fixed odd prime $p$ and $\varepsilon>0$,  there is a constant $C_{p, \varepsilon}$ such that 
		\[
		a(3\text{-}\AP,\F_p^n)\le C_{p, \varepsilon}\,n^{1+\varepsilon}\,p^{n/2}
		=p^{n/2+o(n)}
		\] holds for the  minimum size  $a(3\text{-}\AP,\F_p^n)$ of a complete 3-AP-free subset of $\F_p^n$,  
		for all $n\ge1$.  %This gives an all-dimension vector-space result for ordinary complete $3$-AP-free sets, in the direction of a question posed by Csajb\'ok and Nagy (2024, JCTA).
		
		%The cyclic construction is based on flexible binary interval blocks.  Given integers $d_1=1,d_2,\ldots,d_n$ satisfying
		%\[       2S_i+1\le d_{i+1}\le 3S_i+1,     \qquad S_i=d_1+\cdots+d_i,\]
		%the subset-sum set $\{\sum\varepsilon_i d_i:\varepsilon_i\in\{0,1\}\}$ is an integer block which is both $(2,-1)$-avoiding and interval-saturating.  Varying the digits fills all moduli needed between powers of $3$ and powers of $4$.  The vector-space construction uses two-sided complete quadratic graphs over a subfield flag, followed by a padding argument.
	\end{abstract}
	
	
	\section{Introduction}
	
	Let $G$ be an abelian group, written additively.  A $3$-term arithmetic progression, or $3$-AP, is a set of three distinct elements of the form
	\[
	g,\quad g+d,\quad g+2d .
	\]
	A set is $3$-AP-free if it contains no such progression, and it is complete $3$-AP-free if it is maximal with respect to this property.  The maximum-size problem for $3$-AP-free sets is a central theme of additive combinatorics, going back to the Salem-Spencer and the Behrend construction \cite{Behrend1946, Salem} and Roth's theorem \cite{Roth1953} over the integers, Meshulam's finite-group bounds \cite{Meshulam1995}, and, in vector spaces, the polynomial-method breakthroughs of Croot--Lev--Pach \cite{CLP2017} and Ellenberg--Gijswijt \cite{EG2017}.  The general abelian-group setting was studied by Frankl, Graham and R\"odl \cite{FGR1987}; see also Shkredov's survey \cite{Shkredov2006}.
	
	This paper concerns the opposite end of the spectrum: how small can a complete progression-free set be?  This question is naturally a saturation problem.  A $3$-AP-free set $A$ is complete precisely when every point $x\in G\setminus A$ lies in a $3$-term progression together with two points of $A$.  In a previous paper, the authors \cite{CsajbokNagy2025} developed this viewpoint for vector spaces, cyclic groups and more general abelian groups.  In particular, they introduced $W$-avoiding and $W$-saturating sets, proved direct-product mechanisms for fixed coefficient vectors, and related the problem to complete caps and saturating sets in finite affine spaces.
	
	Several tight or near-tight cases were already obtained there.  For example, algebraic constructions in two-dimensional finite vector spaces give complete $3$-AP-free sets of square-root size under explicit nonsquare hypotheses, and direct-product methods extend these to further infinite families of dimensions. The authors  \cite[Problem~5.1]{CsajbokNagy2025} asked whether the natural square-root lower bound is tight up to an absolute constant in every vector space.  Our second main result gives an upper bound $C_{p, \varepsilon} n^{1+\varepsilon} p^{n/2}=p^{n/2+o(n)}$ for ordinary complete $3$-AP-free sets over each fixed odd prime field $\F_p$, for all dimensions.  The stronger absolute-constant form  $C'_{p}  p^{n/2}$ remains open.  The same paper also posed the corresponding cyclic-group problem, and proved an upper bound for a positive fraction of the integers which is tight up to a small  multiplicative constant.
	
	The connection with finite geometry is especially transparent in vector spaces.  A cap in an affine space is a point set meeting each line in at most two points; it is complete if it cannot be enlarged while preserving this property.  Thus complete caps are geometric counterparts of complete $3$-AP-free sets.  Saturating sets have the complementary covering property that every outside point is incident with a secant determined by two points of the set.  The authors make this correspondence explicit in \cite[Proposition~2.10]{CsajbokNagy2025}.  Complete caps and saturating sets are also central in finite geometry and coding theory, see Giulietti's survey \cite{Giulietti2013} and the references therein.  %For recent progress on small complete capsets over $\F_3$, see Grace and Voloch \cite{GraceVoloch2026}.
	
	The construction below uses  complete $(2,-1)$-avoiding sets.  For $A\subseteq G$ put
	\[
	2A\dotminus A=\{2a-b:\ a,b\in A,\ a\ne b\}.
	\]
	We say that $A$ is $(2,-1)$-avoiding  if
	\[
	A\cap(2A\dotminus A)=\emptyset,
	\]
	and complete $(2,-1)$-avoiding (in $G$) if, in addition,
	\[
	G=A\cup(2A\dotminus A).
	\]
	Thus no non-trivial relation $c=2a-b$ occurs inside $A$, and every outside point is covered as an endpoint of a progression whose other endpoint and midpoint lie in $A$.  This is stronger than completeness of $3$-AP-free sets:  the latter completion allows the outside point to appear either as an endpoint, $x=2a-b$, or as a midpoint, $2x=a+b$.  This will be exploited later on.%The advantage of the stronger condition is that it behaves well under  product constructions, as  demonstated by the authors in \cite[Proposition~2.3 and Corollary~2.4]{CsajbokNagy2025}.
	
	\begin{definition}
		Let $a(3\text{-}\AP,G)$ denote the minimum size of a complete $3$-AP-free subset of $G$, and let $\aavo(G)$ denote the minimum size of a complete $(2,-1)$-avoiding subset of $G$, with the convention $\aavo(G)=\infty$ if no such set exists.     
	\end{definition}
	The trivial counting bound gives
	\[
	\aavo(G)\ge (1+o(1))\sqrt{|G|},
	\]
	since $2A\dotminus A$ has at most $|A|(|A|-1)$ elements.  Our first main theorem gives an explicit upper bound with constant $2$ for every cyclic group.
	
	\begin{theorem}
		\label{thm:main}
		For every integer $m\ge1$ there exists a complete $3$-AP-free set $A\subseteq\Z_m$ such that
		\[
		|A|< 2\sqrt m .
		\]
		Moreover, for every $m\ge81$ the set $A$ may be chosen complete $(2,-1)$-avoiding.  In particular,
		\[
		a(3\text{-}\AP,\Z_m)< 2\sqrt m
		\qquad (m\ge1)
		\]
		and
		\[
		\aavo(\Z_m)<2\sqrt m
		\qquad (m\ge81).
		\]
	\end{theorem}
	The authors \cite{CsajbokNagy2025} showed a slightly stronger result 
	$a(3\text{-}\AP,\Z_m)\le \frac32\sqrt m$, when $m$ satisfied the relation $\frac{2}{3}4^t<m\le4^t$ for some $t \in \mathbb{N}$. 
	We remark here that there is an unexpected obstacle for proving an extension for all $m$: there are a few values of $m$ for which complete $(2, -1)$ avoiding sets does not exist at all, see Section \ref{sec:small-moduli}.
	
	Concerning the vector-space variant of the problem, several  constructions were obtained by the authors \cite{CsajbokNagy2025} for structured families of finite vector spaces, using conic-type cap constructions and direct-product methods, which matched the order of magnitude of the lower bound. Some technical restrictions were required  though, either on the dimension or on the order of the field.
	These results show that the trivial lower bound is essentially sharp in a number of natural cases, but they did not provide a general upper bound   $p^{n/2+o(n)}$ over a fixed field $\F_p$, when $-2$ is a square element in $\F_p$. In fact, one of the concluding problems of~\cite{CsajbokNagy2025} asked whether the square-root lower bound $p^{n/2}$ for complete \(3\)-AP-free sets in vector spaces is sharp up to a constant factor uniformly in the dimension. In characteristic \(3\), the recent algebraic capset construction of Grace and Voloch~\cite{GraceVoloch2026} gives complete capsets of size  $2\cdot (3^{n/2}-1)$ for $n$ even. Our second main result is the following.
	
	\begin{theorem}
		\label{thm:vector-main}
		
		Let \(p\) be an odd prime and let \(\varepsilon>0\).  Then there is a constant
		\(C_{p,\varepsilon}\) such that, for every \(n\ge1\),
		\[
		a(3\text{-}\mathrm{AP},\mathbb F_p^n)
		\le
		C_{p,\varepsilon}\,n^{1+\varepsilon}p^{n/2}.
		\]
		In particular, for every fixed odd prime \(p\),
		\[
		a(3\text{-}\mathrm{AP},\mathbb F_p^n)
		=
		p^{n/2+o(n)}.
		\]
	\end{theorem}
	
	%The cyclic theorem answers the square-root question with the explicit absolute constant $2$, and suggests that flexible interval-saturating blocks, rather than fixed-base digit blocks alone, are the right objects to study. The earlier cyclic target constant $3/2$ remains open. The vector-space theorem is of a different nature: it uses algebraic graphs over extension fields to obtain a near-square-root bound in every dimension over a fixed odd prime field.
	
	We briefly outline the proofs.  To show Theorem \ref{thm:main}, we first record the implication from complete $(2,-1)$-avoidance to complete $3$-AP-freeness.  Second, we introduce $R$-complete sets (in \cref{sec: blocks}): integer sets contained in $[0,R]$ that are complete $(2,-1)$-avoiding in the whole interval $[-R,2R]\subset \Z$.  Such a sets descends to a complete $(2,-1)$-avoiding set in $\Z_m$ whenever $2R<m\le 3R+1$.  This enables us to generalize the former base-$4$  construction, and allows the application of several moduli, provided $m\ge 81$ to complete the proof of theorem \cref{thm:main} in Section \ref{sec:binary}.  A finite check handles the small cases $m<81$.  Finally,  Section \ref{sec:vector-spaces} proves \cref{thm:vector-main}.
	Graphs of quadratic functions 
	give complete  $(2,-1)$-avoiding sets in many special dimensions.  These special sets can be multiplied by complete $3$-AP-free sets in the remaining coordinates, to  obtain a general upper  bound over each fixed odd prime field in any dimension.

	\section{Complete $(2,-1)$-avoidance and complete $3$-AP-freeness}
	
	We recall the connection between  complete $(2,-1)$-avoiding and complete $3$-AP-free sets, that will be used throughout the paper, cf. \cite{CsajbokNagy2025}.
	
	\begin{proposition}\label{prop:relation}
		Let $G$ be an abelian group and $A\subseteq G$.
		\begin{enumerate}[label=\textup{(\roman*)}]
			\item If $A$ is $(2,-1)$-avoiding, then $A$ is $3$-AP-free.
			\item If $A$ is complete $(2,-1)$-avoiding, then $A$ is complete $3$-AP-free.
			\item If $G$ has odd order, then $(2,-1)$-avoidance is equivalent to the usual avoidance of non-trivial $3$-APs.
		\end{enumerate}
	\end{proposition}
	
	\begin{proof}
		If $A$ contained a $3$-AP $b,a,c$ with distinct elements, then $c=2a-b$ with $a,b\in A$ and $a\ne b$, contradicting $(2,-1)$-avoidance.  This proves (i).
		
		For (ii), let $x\in G\setminus A$.  Since $A$ is complete $(2,-1)$-avoiding, there exist $a,b\in A$, $a\ne b$, such that $x=2a-b$.  Then $b,a,x$ form a $3$-term arithmetic progression.  Since $x\notin A$ and $a\ne b$, adjoining $x$ to $A$ creates a $3$-AP.  Hence $A$ is complete $3$-AP-free.
		
		For (iii), the only possible discrepancy between a relation $c=2a-b$ with $a\ne b$ and a relation involving three distinct elements occurs when $c=b$.  In that case $2(a-b)=0$.  If $G$ has odd order, multiplication by $2$ is injective, so $a=b$, a contradiction.  Thus every non-trivial relation $c=2a-b$ has three distinct elements.
	\end{proof}

	\subsection{Small moduli and  $(2,-1)$ completion}\label{sec:small-moduli}
	
	The main invariant in this paper is $a(3\text{-}\AP,G)$.  The auxiliary invariant $\aavo(G)$ is useful because, by Proposition~\ref{prop:relation},
	\[
	a(3\text{-}\AP,G)\le \aavo(G)
	\]
	whenever $\aavo(G)$ is finite.  For the finite range not covered by the binary construction of Section \ref{sec:cyclic}, we used an exhaustive verification to determine the values of     $a(3\text{-}\AP,\Z_m)$ and $\aavo(\Z_m).$  The following table records
	\[
	f(m)=a(3\text{-}\AP,\Z_m),
	\qquad
	g(m)=\aavo(\Z_m),
	\]
	for $m\le80$, with $g(m)=\infty$ if no complete $(2,-1)$-avoiding subset of $\Z_m$ exists.

	Explicit attaining sets are listed in \cref{app:small-table}. Exactness of the displayed values, and the non-existence entries in the
	\(g(m)\)-column, were verified by exhaustive enumeration.
	
	% \begin{center}
		% {\scriptsize
			% \setlength{\fboxsep}{4pt}
			% \renewcommand{\arraystretch}{0.92}
			% \begin{tabular}{@{}c@{\hspace{0.8em}}c@{\hspace{0.8em}}c@{\hspace{0.8em}}c@{}}
				% \fbox{\begin{tabular}{c|c|c}
						% $m$ & $f(m)$ & $g(m)$\\ \hline
						% 1&1&1\\
						% 2&2&$\infty$\\
						% 3&2&2\\
						% 4&2&2\\
						% 5&2&$\infty$\\
						% 6&4&$\infty$\\
						% 7&3&3\\
						% 8&4&$\infty$\\
						% 9&4&4\\
						% 10&4&4\\
						% 11&4&4\\
						% 12&4&4\\
						% 13&4&4\\
						% 14&4&4\\
						% 15&4&4\\
						% 16&4&4\\
						% 17&4&$\infty$\\
						% 18&4&$\infty$\\
						% 19&5&$\infty$\\
						% 20&4&$\infty$
						% \end{tabular}}
				% &
				% \fbox{\begin{tabular}{c|c|c}
						% $m$ & $f(m)$ & $g(m)$\\ \hline
						% 21&5&5\\
						% 22&6&6\\
						% 23&6&6\\
						% 24&6&6\\
						% 25&6&6\\
						% 26&6&6\\
						% 27&6&7\\
						% 28&6&6\\
						% 29&6&7\\
						% 30&6&7\\
						% 31&6&7\\
						% 32&6&8\\
						% 33&6&8\\
						% 34&7&7\\
						% 35&6&8\\
						% 36&7&8\\
						% 37&7&8\\
						% 38&7&8\\
						% 39&6&8\\
						% 40&8&8
						% \end{tabular}}
				% &
				% \fbox{\begin{tabular}{c|c|c}
						% $m$ & $f(m)$ & $g(m)$\\ \hline
						% 41&8&8\\
						% 42&8&8\\
						% 43&8&8\\
						% 44&8&8\\
						% 45&8&8\\
						% 46&8&8\\
						% 47&8&8\\
						% 48&8&8\\
						% 49&8&8\\
						% 50&8&8\\
						% 51&8&8\\
						% 52&8&8\\
						% 53&8&8\\
						% 54&8&8\\
						% 55&8&8\\
						% 56&8&8\\
						% 57&8&8\\
						% 58&8&8\\
						% 59&8&8\\
						% 60&8&8
						% \end{tabular}}
				% &
				% \fbox{\begin{tabular}{c|c|c}
						% $m$ & $f(m)$ & $g(m)$\\ \hline
						% 61&8&8\\
						% 62&8&8\\
						% 63&8&8\\
						% 64&8&8\\
						% 65&9&10\\
						% 66&10&10\\
						% 67&10&11\\
						% 68&10&10\\
						% 69&10&10\\
						% 70&10&11\\
						% 71&10&10\\
						% 72&10&11\\
						% 73&10&10\\
						% 74&11&11\\
						% 75&10&11\\
						% 76&10&12\\
						% 77&11&12\\
						% 78&11&11\\
						% 79&11&12\\
						% 80&10&12
						% \end{tabular}}
				% \end{tabular}
			% }
		% \end{center}
	
	\begin{center}
		{\scriptsize
			\setlength{\fboxsep}{4pt}
			\renewcommand{\arraystretch}{0.92}
			\begin{tabular}{@{}c@{\hspace{0.8em}}c@{\hspace{0.8em}}c@{\hspace{0.8em}}c@{}}
				\fbox{\begin{tabular}{c|c|c}
						$m$ & $f(m)$ & $g(m)$\\ \hline
						1&1&1\\
						2&2&$\infty$\\
						3&2&2\\
						4&2&2\\
						5&2&$\infty$\\
						6&4&$\infty$\\
						7&3&3\\
						8&4&$\infty$\\
						9&4&4\\
						10&4&4\\
						11&4&4\\
						12&4&4\\
						13&4&4\\
						14&4&4\\
						15&4&4\\
						16&4&4\\
						17&4&$\infty$\\
						18&4&$\infty$\\
						19&5&$\infty$\\
						20&4&$\infty$
				\end{tabular}}
				&
				\fbox{\begin{tabular}{c|c|c}
						$m$ & $f(m)$ & $g(m)$\\ \hline
						21&5&5\\
						22&6&6\\
						23&6&6\\
						24&6&6\\
						25&6&6\\
						26&6&6\\
						27&6&7\\
						28&6&6\\
						29&6&7\\
						30&6&7\\
						31&6&7\\
						32&6&8\\
						33&6&8\\
						34&7&7\\
						35&6&8\\
						36&7&8\\
						37&7&8\\
						38&7&8\\
						39&6&8\\
						40&8&8
				\end{tabular}}
				&
				\fbox{\begin{tabular}{c|c|c}
						$m$ & $f(m)$ & $g(m)$\\ \hline
						41&7&8\\
						42&7&8\\
						43&7&8\\
						44&7&8\\
						45&8&8\\
						46&8&8\\
						47&8&8\\
						48&8&8\\
						49&8&8\\
						50&8&8\\
						51&8&8\\
						52&8&8\\
						53&8&8\\
						54&8&8\\
						55&8&8\\
						56&8&8\\
						57&8&8\\
						58&8&8\\
						59&8&8\\
						60&8&8
				\end{tabular}}
				&
				\fbox{\begin{tabular}{c|c|c}
						$m$ & $f(m)$ & $g(m)$\\ \hline
						61&8&8\\
						62&8&8\\
						63&8&8\\
						64&8&8\\
						65&8&10\\
						66&8&10\\
						67&8&10\\
						68&8&10\\
						69&9&10\\
						70&8&10\\
						71&10&10\\
						72&8&10\\
						73&9&9\\
						74&10&11\\
						75&10&11\\
						76&10&10\\
						77&10&11\\
						78&10&11\\
						79&10&12\\
						80&10&11
				\end{tabular}}
			\end{tabular}
		}
	\end{center}
	
	Thus complete $3$-AP-free sets of size less than $2\sqrt m$ exist for every $m<81$.  Complete $(2,-1)$-avoiding sets may fail to exist for small moduli; among $m\le80$, the non-existence cases displayed here are
	\[
	m\in\{2,5,6,8,17,18,19,20\}.
	\]
	The construction in \cref{sec:binary} covers all $m\ge81$.
	\section{Complete $3$-AP-free sets in cyclic groups} \label{sec:cyclic}
	
	\subsection{$R$-complete sets and reduction modulo $m$}\label{sec: blocks}
	
	All intervals in this section are intervals of integers.  If $X,Y\subseteq\Z$, then
	\[
	2X-Y=\{2x-y:\ x\in X,\ y\in Y\},
	\]
	and
	\[
	2X\dotminus Y=\{2x-y:\ x\in X,\ y\in Y,\ x\ne y\}.
	\]
	When $X=Y$ we write $2X\dotminus X$.
	
	\begin{definition}\label{def:block}
		%Let $N\equiv1\pmod 3$ and put $R=(N-1)/3$. 
		A set $A\subseteq[0,R]$ is called an \emph{$R$-complete set} if
		\[
		A\cap(2A\dotminus A)=\emptyset
		\quad\text{and}\quad
		A\cup(2A\dotminus A)=[-R,2R].
		\]
		%Its weight is $\wt(A)=|A|$.
	\end{definition}
	
	The definition captures exactly what is needed for the modular construction: a small-diameter set which is internally $(2,-1)$-avoiding and which saturates the whole interval determined by its diameter.
	
	\begin{lemma}\label{lem:modreduction}
		Let $A\subseteq[0,R]$ be an $R$-complete set.  If
		$  2R<m\le 3R+1$,
		then the image of $A$ in $\Z_m$ is complete $(2,-1)$-avoiding.
	\end{lemma}
	
	\begin{proof}
		Since $A\subseteq[0,R]$ and $R<m$, reduction modulo $m$ is injective on $A$.
		
		First prove avoidance.  Suppose that
		$  c\equiv 2a-b \pmod m$
		for $a,b,c\in A$.  Then $2a-b-c$ is a multiple of $m$.  But $a,b,c\in[0,R]$, so
		$$ -2R\le 2a-b-c\le2R.$$
		Since $m>2R$, the only multiple of $m$ in this interval is $0$.  Thus $2a-b-c=0$ over the integers, and the $R$-complete property gives $a=b=c$.
		
		Now prove saturation.  Let $r\in\Z_m\setminus A$ and choose its representative $0\le r\le m-1$.  If $r\le2R$, set $y=r$.  If $r>2R$, set $y=r-m$.  In the second case $y\le -1$, and since $m\le3R+1$,
		\[
		y=r-m\ge2R+1-m\ge -R.
		\]
		Thus $y\in[-R,2R]$ and $y\equiv r\pmod m$.  Moreover $y\notin A$: in the first case this follows from the choice of $r$, and in the second case $y<0$ whereas $A\subseteq[0,R]$.  By the $R$-complete property, $y=2a-b$ for some distinct $a,b\in A$.  Hence $r\equiv2a-b\pmod m$.
	\end{proof}
	
	% Equivalently, since $R=(N-1)/3$, the admissible interval for $m$ is
	% \[
	%         \frac{2(N-1)}3<m\le N.
	% \]
	
	\subsection{$R$-complete sets}\label{sec:binary}
	
	\begin{lemma}[Binary $R$-complete sets]\label{lem:binaryinterval}
		Let $n\ge1$.  Let $d_1=1$, put $S_i=d_1+\cdots+d_i$, and suppose that
		\[
		2S_i+1\le d_{i+1}\le 3S_i+1
		\qquad (i=1,\ldots,n-1).
		\]
		Then
		\[
		P=\left\{\sum_{i=1}^n \varepsilon_i d_i:
		\varepsilon_i\in\{0,1\}\right\}
		\]
		is an $S_n$-complete set.  Moreover $|P|=2^n$.
	\end{lemma}
	
	\begin{proof}
		First, the subset sums are distinct.  Indeed, if
		\[
		\sum_{i=1}^n \delta_i d_i=0,
		\qquad \delta_i\in\{-1,0,1\},
		\]
		and not all $\delta_i$ are zero, let $j$ be the largest index with $\delta_j\ne0$.  Then the contribution of the lower indices has absolute value at most $S_{j-1}$, while $d_j>S_{j-1}$, a contradiction.  Hence $|P|=2^n$.
		
		Now prove avoidance.  Suppose $c=2a-b$ with $a,b,c\in P$.  Write
		\[
		a=\sum \alpha_i d_i,\qquad
		b=\sum \beta_i d_i,\qquad
		c=\sum \gamma_i d_i,
		\qquad
		\alpha_i,\beta_i,\gamma_i\in\{0,1\}.
		\]
		Then
		\[
		0=\sum_{i=1}^n (2\alpha_i-\beta_i-\gamma_i)d_i .
		\]
		If not all triples $(\alpha_i,\beta_i,\gamma_i)$ are equal, let $j$ be the largest index for which they are not equal.  Then
		\[
		2\alpha_j-\beta_j-\gamma_j\in\{-2,-1,1,2\}.
		\]
		The lower-index contribution has absolute value at most $2S_{j-1}$, while $d_j>2S_{j-1}$.  Thus cancellation is impossible.  Hence all triples are equal, and so $a=b=c$.
		
		It remains to prove saturation.  We prove by induction on $i$ that every integer in $[-S_i,2S_i]$ has a representation
		\[
		\sum_{h=1}^i \eta_h d_h,
		\qquad
		\eta_h\in\{-1,0,1,2\}.
		\]
		For $i=1$ this is immediate.  Suppose the claim holds for $S_i$, and write $d=d_{i+1}$.  The four possible new digit values give the intervals
		\[
		-d+[-S_i,2S_i],\quad
		[-S_i,2S_i],\quad
		d+[-S_i,2S_i],\quad
		2d+[-S_i,2S_i].
		\]
		They are consecutive because $d\le3S_i+1$.  Their union is therefore
		\[
		[-(S_i+d),\,2(S_i+d)].
		\]
		Thus every element of $[-S_n,2S_n]$ is of the form $2p-q$ with $p,q\in P$.  If the element is not in $P$, then necessarily $p\ne q$.  Hence $P$ is an $S_n$-complete set.
	\end{proof}
	
	\begin{remark}\label{rem:basefour}
		The earlier base-$4$ construction in \cite{CsajbokNagy2025} of the authors  is recovered as the extreme special case of \cref{lem:binaryinterval}. Indeed, choose
		$d_i=4^{i-1}$, $i=1,\ldots,n$.
		Then
		$ S_i=1+4+\cdots+4^{i-1}=\frac{4^i-1}{3},$
		and the admissibility condition in \cref{lem:binaryinterval} is satisfied with equality at the upper end:
		$ d_{i+1}=4^i=3S_i+1.$
		The resulting subset-sum  set is precisely
		\[
		B_n=\left\{\sum_{i=0}^{n-1}\varepsilon_i4^i:
		\varepsilon_i\in\{0,1\}\right\}.
		\]
		
		Thus the former $(4^n-1)/3$-complete  set is not a separate construction, but the right-endpoint case of the more flexible binary interval construction.  The new feature of \cref{lem:binaryinterval} is that the next digit $d_{i+1}$ may be chosen anywhere in the interval
		$ [2S_i+1,  3S_i+1],$
		which fills all intermediate moduli $3R+1$ with
		$R\in  [ \frac{3^n-1}{2}, \frac{4^n-1}{3}]$.
	\end{remark}
	
	\begin{corollary}\label{cor:manybinaryseeds}
		For every $n\ge1$ and every integer $R$ satisfying
		\[
		\frac{3^n-1}{2}\le R\le \frac{4^n-1}{3},
		\]
		there exists an $R$-complete set  of size $2^n$.
	\end{corollary}
	
	\begin{proof}
		Let $\mathcal S_n$ denote the set of all integers $S_n$ that can occur from a sequence $d_1,\ldots,d_n$ satisfying $d_1=1$, $S_i=d_1+\cdots+d_i$, and
		\[
		2S_i+1\le d_{i+1}\le 3S_i+1
		\qquad (i=1,\ldots,n-1).
		\]
		We prove by induction that
		\[
		\mathcal S_n=
		\left[
		\frac{3^n-1}{2},
		\frac{4^n-1}{3}
		\right]\cap\Z .
		\]
		
		For $n=1$, we have $d_1=1$, hence $S_1=1$, and indeed
		$\frac{3^1-1}{2}=\frac{4^1-1}{3}=1.$
		
		Assume the claim holds for $n$.  Fix $S\in\mathcal S_n$.  The next digit $d_{n+1}$ may be chosen arbitrarily in the integer interval
		\[
		2S+1\le d_{n+1}\le 3S+1.
		\]
		Therefore the next sum
		$ S_{n+1}=S+d_{n+1}$
		may be any integer in
		$ [3S+1,4S+1].$
		
		Thus
		\[
		\mathcal S_{n+1}
		=
		\bigcup_{S\in\mathcal S_n}
		[3S+1,4S+1]\cap\Z .
		\]
		
		By the induction hypothesis, $S$ runs through the consecutive interval
		\[
		\left[
		\frac{3^n-1}{2},
		\frac{4^n-1}{3}
		\right]\cap\Z .
		\]
		For consecutive values $S$ and $S+1$, the corresponding intervals are
		\[
		[3S+1,4S+1]
		\quad\text{and}\quad
		[3S+4,4S+5].
		\]
		These intervals have no gap whenever
		\[
		3S+4\le4S+2,
		\]
		that is, whenever $S\ge2$.  In the first induction step there is only the single value $S=1$, and after that the lower endpoint is at least $4$.  Hence the union is one consecutive integer interval.
		
		Its lower endpoint is
		\[
		3\cdot\frac{3^n-1}{2}+1
		=
		\frac{3^{n+1}-1}{2},
		\]
		and its upper endpoint is
		\[
		4\cdot\frac{4^n-1}{3}+1
		=
		\frac{4^{n+1}-1}{3}.
		\]
		Thus
		\[
		\mathcal S_{n+1}
		=
		\left[
		\frac{3^{n+1}-1}{2},
		\frac{4^{n+1}-1}{3}
		\right]\cap\Z .
		\]
		This completes the induction.
		
		Now let $R$ be any integer in the stated interval.  Then $R\in\mathcal S_n$, so there are integers $d_1,\ldots,d_n$ satisfying the hypotheses of \cref{lem:binaryinterval} and with $S_n=R$.  By \cref{lem:binaryinterval}, the set
		\[
		P=\left\{\sum_{i=1}^n \varepsilon_i d_i:
		\varepsilon_i\in\{0,1\}\right\}
		\]
		is an $R$-complete set  and has size  $|P|=2^n$.
	\end{proof}
	
	\begin{theorem}\label{thm:threshold81}
		For every integer $m\ge81$, there exists a complete $(2,-1)$-avoiding set $A\subseteq\Z_m$ such that
		\[
		|A|<2\sqrt m .
		\]
		Consequently,
		\[
		a(3\text{-}\AP,\Z_m)
		\le
		\aavo(\Z_m)
		<
		2\sqrt m
		\qquad (m\ge81).
		\]
	\end{theorem}
	
	\begin{proof}
		Let $m\ge81$.  Choose $n\ge4$ such that
		$ 4^{n-1}<m\le4^n.$
		We first note that
		$3^n\le m\le4^n.$
		Indeed, for $n=4$ this follows from $m\ge81=3^4$.  For $n\ge5$, we have
		$3^n\le4^{n-1}<m,$ so again $3^n\le m$.
		
		By \cref{cor:manybinaryseeds}, for every integer
		\[
		R\in
		\left[
		\frac{3^n-1}{2},
		\frac{4^n-1}{3}
		\right]
		\]
		there exists an $R$-complete set of size $2^n$.
		
		Now consider the intervals
		\[
		[2R+1,3R+1]
		\]
		as $R$ runs through
		\[
		\left[
		\frac{3^n-1}{2},
		\frac{4^n-1}{3}
		\right]\cap\Z .
		\]
		For consecutive values $R$ and $R+1$, these intervals are
		\[
		[2R+1,3R+1]
		\quad\text{and}\quad
		[2R+3,3R+4].
		\]
		They have no gap because
		$ 2R+3\le3R+2$ 
		for every $R\ge1$.  Therefore their union is a single integer interval.  Its lower endpoint is
		\[
		2\cdot\frac{3^n-1}{2}+1=3^n,
		\]
		and its upper endpoint is
		\[
		3\cdot\frac{4^n-1}{3}+1=4^n.
		\]
		Hence the intervals $[2R+1,3R+1]$ cover all integers in
		$ [3^n,4^n].$
		
		Since $m\in[3^n,4^n]$, we may choose an integer $R$ in the above range such that
		\[
		2R+1\le m\le3R+1.
		\]
		Equivalently, $2R<m\le3R+1$.  Let $P\subseteq[0,R]$ be the $R$-complete set of size $2^n$ given by \cref{cor:manybinaryseeds}.  By \cref{lem:modreduction}, the image of $P$ in $\Z_m$ is complete $(2,-1)$-avoiding.
		
		Its size is $|P|=2^n$.  Since $4^{n-1}<m$, we have $2^{n-1}<\sqrt m$.  Multiplying by $2$, we obtain
		\[
		|P|=2^n<2\sqrt m.
		\]
		This proves the theorem.
	\end{proof}
	
	\begin{proof}[Proof of \cref{thm:main}]
		For $m\ge81$, this is \cref{thm:threshold81}.  For $m<81$, the claim follows from the finite verification displayed in \cref{sec:small-moduli} and certified in \cref{app:small-table}.  Therefore $a(3\text{-}\AP,\Z_m)<2\sqrt m$ for every $m\ge1$.  The stronger complete $(2,-1)$-avoiding assertion for $m\ge81$ is exactly \cref{thm:threshold81}.
	\end{proof}

	\section{Complete $3$-AP-free sets in vector spaces} \label{sec:vector-spaces}
	
	%The cyclic construction above is independent of the vector-space question from \cite[Problem~5.1]{CsajbokNagy2025}.  
	Here we prove   a near-square-root bound for complete
	$3$-AP-free sets in every sufficiently large dimension over a fixed odd prime
	field, Theorem \ref{thm:vector-main}. Note that our construction is not a complete $(2,-1)$-avoiding construction in
	all dimensions; the final padding step uses  complete $3$-AP-free sets.

	\subsection{Endpoint-midpoint complete  sets and products}
	
	\begin{definition}[Endpoint-midpoint completeness]
		Let $G$ be an abelian group of odd order.  We shall say that a $3$-AP-free set
		$A\subseteq G$ is \emph{endpoint-midpoint complete } if every $x\in G\setminus A$ is
		covered in both possible ways: there both exist  distinct  pairs  $a,b\in A$ and $c,d\in A$ such that
		\[
		x=2a-b, \quad 2x=c+d.
		\]
		
	\end{definition}
	
	Thus every outside point can be used both as an endpoint and as the midpoint of
	a $3$-term arithmetic progression whose other two points lie in $A$.
	
	\begin{lemma}\label{lem:twosided-product}
		Let $G$ and $H$ be abelian groups of odd order.  Suppose that
		$A\subseteq G$ is two-sided $3$-AP-free and that
		$B\subseteq H$ is complete $3$-AP-free.  Then
		\[
		A\times B\subseteq G\times H
		\]
		is complete $3$-AP-free.
	\end{lemma}
	
	\begin{proof}
		The product is $3$-AP-free: a non-trivial progression in $A\times B$ would have
		a non-trivial projection in at least one coordinate, contradicting the
		$3$-AP-freeness of the corresponding factor.
		
		Now let $(x,y)\in (G\times H)\setminus(A\times B)$.  If $x\notin A$ and
		$y\in B$, use a completion of $x$ in $A$ and keep the second
		coordinate fixed at $y$.  If $x\in A$ and $y\notin B$, use a completion of $y$
		in $B$ and keep the first coordinate fixed at $x$.
		
		Finally suppose that $x\notin A$ and $y\notin B$.  Since $B$ is complete, the
		point $y$ is either an endpoint or the midpoint of a progression determined by
		two points of $B$.  If $y=2b_1-b_2$, choose an endpoint representation
		$x=2a_1-a_2$ in $A$.  If $2y=b_1+b_2$, choose a midpoint representation
		$2x=a_1+a_2$ in $A$.  In either case the two product points
		$(a_1,b_1)$ and $(a_2,b_2)$ form a $3$-term progression with $(x,y)$.  Hence the
		product is complete.
	\end{proof}
	
	\subsection{Quadratic graphs} Suppose that $p$ is an odd prime.
	Throughout this subsection, a map \(Q:U\to W\) between \(\mathbb F_p\)-vector
	spaces is called a \textit{homogeneous quadratic map} if
	\[
	Q(\lambda u)=\lambda^2 Q(u)
	\qquad(\lambda\in\mathbb F_p,\ u\in U)
	\]
	and the polar map
	\[
	B_Q(u,v)=\frac{Q(u+v)-Q(u)-Q(v)}{2}
	\]
	is bilinear. This implies the identities
	\[
	Q(u+v)=Q(u)+2B_Q(u,v)+Q(v)
	\]
	and
	\[
	Q(u+t)+Q(u-t)=2Q(u)+2Q(t),
	\]
	which will be used repeatedly below.
	
	The next lemma is the core of our construction. It shows that graphs of quadratic maps can ensure $3$-AP-freeness and saturation at the same time. Moreover, their endpoint-midpoint completeness property will enable us to extend the construction to higher dimension as well.
	%a strengthened form of the  quadratic graph criterion.  The endpoint part gives complete $(2,-1)$-avoidance, while the midpoint part is what will allow us to multiply by a complete $3$-AP-free set in the remaining coordinates.
	
	\begin{lemma}[Quadratic graph criterion]\label{lem:quadratic-graph-twosided}
		Let $p$ be odd, let $U,W$ be vector spaces over $\F_p$, and let
		$Q:U\to W$ be a homogeneous quadratic map.  Suppose that
		\[
		Q^{-1}(0)=\{0\}
		\qquad\text{and}\qquad
		Q(U)=W .
		\]
		Then
		\[
		A_Q=\{(u,Q(u)):u\in U\}\subseteq U\oplus W
		\]
		is endpoint-midpoint complete  $3$-AP-free.  %Moreover, it is complete
		%$(2,-1)$-avoiding.
	\end{lemma}
	
	\begin{proof}
		First prove avoidance.  Suppose that
		\[
		(w,Q(w))=2(u,Q(u))-(v,Q(v))
		\]
		with $u,v,w\in U$.  The first coordinate gives $w=2u-v$.  Write
		$v=u-t$, so $w=u+t$.  The second coordinate gives
		\[
		Q(u+t)=2Q(u)-Q(u-t),
		\]
		or equivalently
		\[
		Q(u+t)+Q(u-t)=2Q(u).
		\]
		Since $Q$ is quadratic,
		\[
		Q(u+t)+Q(u-t)=2Q(u)+2Q(t).
		\]
		As $p$ is odd, this gives $Q(t)=0$, hence $t=0$.  Therefore $u=v=w$.
		Thus $A_Q$ is $(2,-1)$-avoiding, and in particular $3$-AP-free.
		
		Let $(u,Q(u)+h)\notin A_Q$, so $h\ne0$.  For endpoint-saturation, choose
		$t\in U$ with
		\[
		Q(t)=-h/2 .
		\]
		Then $t\ne0$.  A direct quadratic calculation gives
		\[
		2Q(u+t)-Q(u+2t)=Q(u)-2Q(t)=Q(u)+h,
		\]
		and hence
		\[
		(u,Q(u)+h)
		=2(u+t,Q(u+t))-(u+2t,Q(u+2t)).
		\]
		The two graph points are distinct, so $A_Q$ is complete $(2,-1)$-avoiding.
		
		For midpoint-saturation, choose $t\in U$ with $Q(t)=h$.  Again $t\ne0$, and
		\[
		\frac{Q(u+t)+Q(u-t)}2=Q(u)+Q(t)=Q(u)+h.
		\]
		Therefore
		\[
		(u,Q(u)+h)
		=\frac{(u+t,Q(u+t))+(u-t,Q(u-t))}{2},
		\]
		with two distinct graph points.  Thus $A_Q$ is endpoint-midpoint complete.
	\end{proof}
	
	\subsection{A subfield quadratic graph construction}
	
	We now construct quadratic maps satisfying the two hypotheses of
	\cref{lem:quadratic-graph-twosided}.  The construction uses a subfield
	$T\subseteq E$ of odd extension degree.
	
	\begin{theorem}[Subfield quadratic graph construction]\label{thm:subfieldflag}
		Let $p$ be an odd prime.  Let $L\ge1$ be odd, and let $k\ge1$ satisfy
		\[
		p^k>(2L-1)^2 .
		\]
		Put
		\[
		T=\F_{p^k},
		\qquad
		E=\F_{p^{kL}} .
		\]
		Choose a nonsquare $\delta\in E^*$.  Define
		\[
		Q:E\oplus T\longrightarrow E,
		\qquad
		Q(x,s)=x^2-\delta s^2 .
		\]
		Then $Q^{-1}(0)=\{(0,0)\}$ and $Q(E\oplus T)=E$.  Consequently
		\[
		A_Q=\{(x,s,Q(x,s)):x\in E,\ s\in T\}
		\subseteq (E\oplus T)\oplus E
		\]
		is a endpoint-midpoint complete  $3$-AP-free set. Regarded as a subset of $\F_p^{k(2L+1)}$, it has size
		\[
		|A_Q|=p^{k(L+1)} .
		\]
	\end{theorem}
	
	\begin{proof}
		First suppose $Q(x,s)=0$.  Then
		$ x^2=\delta s^2.$
		If $s=0$, then $x=0$.  If $s\ne0$, then
		$\delta=(x/s)^2,$ 
		contradicting the choice of $\delta$ as a nonsquare in $E$.  Hence
		$Q^{-1}(0)=\{(0,0)\}$.
		
		It remains to prove that $Q$ is surjective.  Let $h\in E^*$.  We shall find
		$s\in T$ such that $h+\delta s^2$ is either zero or a square in $E$; then
		$h=x^2-\delta s^2$ for some $x\in E$.
		
		Let $\chi_E$ and $\chi_T$ be the quadratic characters of $E$ and $T$, extended
		by $\chi_E(0)=\chi_T(0)=0$. The norm $N_{E/T}(z)$ of an element $z\in E$ is $z^{(|E|-1)/(|T|-1)}$. Since
		\[
		\chi_E(z)=\chi_T(N_{E/T}(z)),
		\]
		we consider
		\[
		S_h=\sum_{s\in T}\chi_E(h+\delta s^2)
		=\sum_{s\in T}\chi_T\bigl(N_{E/T}(h+\delta s^2)\bigr).
		\]
		Put
		\[
		F_h(X)=(h+\delta X^2)(h^q+\delta^q X^2)\cdots (h^{q^{L-1}}+\delta^{q^{L-1}}X^2) \in T[X],
		\]
		which gives the same $T \rightarrow T$ function as $X \mapsto N_{E/T}(h+\delta X^2)$. 
		Then $\deg F_h\le2L$.  We claim that $F_h$ is not a constant multiple of a
		square in $T[X]$.  Write $\beta=-h/\delta\in E^*$. Over an algebraic closure,
		\[
		F_h(X)=N_{E/T}(\delta)
		\prod_{\sigma\in\mathrm{Gal}(E/T)}(X^2-\sigma(\beta)).
		\]
		If the orbit of $\beta$ under $\mathrm{Gal}(E/T)$ has size $d$, then each
		distinct conjugate occurs with multiplicity $L/d$.  Since $d\mid L$ and $L$ is
		odd, this multiplicity is odd.  Thus $F_h$ has a root of odd multiplicity over
		the algebraic closure, and cannot be a constant multiple of a square.
		
		By the Weil bound for quadratic character sums applied to the 
		polynomial $F_h$ \cite[Theorem~5.41]{LidlNiederreiter1997}, we have
		\[
		|S_h|\le (2L-1)p^{k/2} .
		\]
		The hypothesis $p^k>(2L-1)^2$ gives $|S_h|<p^k=|T|$.  If $h+\delta s^2$ were a
		nonsquare in $E$ for every $s\in T$, then $S_h=-|T|$, a contradiction.  If
		$h+\delta s^2=0$ for some $s$, then $h=0^2-\delta s^2$ is represented; otherwise
		there is an $s\in T$ for which $h+\delta s^2$ is a nonzero square in $E$.  This
		proves $Q(E\oplus T)=E$.
		
		The final assertions follow from \cref{lem:quadratic-graph-twosided}.  The
		ambient dimension is
		\[
		\dim_{\F_p}(E\oplus T\oplus E)=kL+k+kL=k(2L+1),
		\]
		and the graph size is
		\[
		|E||T|=p^{kL}p^k=p^{k(L+1)}.
		\]
	\end{proof}
	
	The surjectivity of $Q$ in Theorem \ref{thm:subfieldflag} can also be
	interpreted as a structured domination statement in Paley graphs.  This
	interpretation is related in spirit to work of Martin and Yip on
	subfields, power residues and generalized Paley graphs
	\cite{MartinYip2025}.  Questions on squares in finite fields and
	their geometric applications go back at least to work of Hirschfeld and
	Sz\H{o}nyi \cite{HirschfeldSzonyi1991}.
	
	\subsection{All dimensions}\label{subsec:all-dimensions}
	
	We now pass from the special dimensions supplied by the subfield quadratic graph
	construction to all dimensions.   Throughout this subsection \(p\) is fixed. Now we reiterate  \cref{thm:vector-main}.
	
	\begin{theorem}
		Let \(p\) be an odd prime and let \(\varepsilon>0\).  Then there is a constant
		\(C_{p,\varepsilon}\) such that, for every \(n\ge1\),
		\[
		a(3\text{-}\mathrm{AP},\mathbb F_p^n)
		\le
		C_{p,\varepsilon}\,n^{1+\varepsilon}p^{n/2}.
		\]
		In particular, for every fixed odd prime \(p\),
		\[
		a(3\text{-}\mathrm{AP},\mathbb F_p^n)
		=
		p^{n/2+o(n)}.
		\]
	\end{theorem}
	
	\begin{proof}
		Choose \(N_0=N_0(p,\varepsilon)\) large enough so that the estimates below hold
		for every \(n\ge N_0\).  We then prove the bound for \(n\ge N_0\) by strong
		induction on \(n\), and finally enlarge \(C_{p,\varepsilon}\) to cover the
		finitely many dimensions \(n<N_0\).
		
		Let \(n\) be large.  Choose \(k\ge2\) minimal with
		\[
		p^k>\left(\frac nk\right)^2 .
		\]
		By minimality,
		\[
		p^{k-1}\le \left(\frac{n}{k-1}\right)^2,
		\]
		and hence
		\[
		p^{k/2}
		\le
		\sqrt p\,\frac{n}{k-1}.
		\tag{5.1}
		\]
		
		Choose an odd integer \(L\ge1\) such that
		$ n_0=k(2L+1)\le n$
		is as large as possible.  Since \(L\) is odd, the admissible values of
		\(2L+1\) differ by \(4\).  Thus the remainder
		$ r=n-n_0$
		satisfies
		\[
		0\le r<4k .
		\tag{5.2}
		\]
		Moreover,
		\[
		2L-1<2L+1\le \frac nk,
		\]
		so
		\[
		p^k>\left(\frac nk\right)^2>(2L-1)^2 .
		\]
		Therefore the subfield quadratic graph construction applies.  It gives a two-sided
		complete \(3\)-AP-free set $A\subseteq \mathbb F_p^{n_0}$
		of size
		\[
		|A|
		=
		p^{k(L+1)}
		=
		p^{n_0/2}p^{k/2}.
		\]
		
		If \(r=0\), then by \((5.1)\),
		\[
		|A|
		\le
		\sqrt p\,\frac{n}{k-1}p^{n/2}.
		\]
		For all sufficiently large \(n\), this is at most
		\[
		n^{1+\varepsilon}p^{n/2},
		\]
		and the remaining finitely many cases are absorbed into the constant
		\(C_{p,\varepsilon}\).
		
		Assume now that \(r>0\).  Since \(r<n\), the induction hypothesis gives a
		complete \(3\)-AP-free set
		$ B\subseteq \mathbb F_p^r$
		with
		\[
		|B|\le
		C_{p,\varepsilon}\,r^{1+\varepsilon}p^{r/2}.
		\]
		Since \(A\) is endpoint-midpoint complete  and \(B\) is complete \(3\)-AP-free, the
		product lemma gives that
		\[
		A\times B
		\subseteq
		\mathbb F_p^{n_0}\times\mathbb F_p^r
		\cong
		\mathbb F_p^n
		\]
		is complete \(3\)-AP-free.  Its size is at most
		\[
		|A||B|
		\le
		C_{p,\varepsilon}\,
		r^{1+\varepsilon}p^{k/2}p^{n/2}.
		\]
		Using \((5.1)\) and \((5.2)\), we get
		\[
		r^{1+\varepsilon}p^{k/2}
		<
		(4k)^{1+\varepsilon}
		\sqrt p\,\frac{n}{k-1}.
		\]
		For all sufficiently large \(n\), the right-hand side is at most
		$   n^{1+\varepsilon}.$
		Indeed, \(k=O_p(\log n)\), so
		\[
		(4k)^{1+\varepsilon}\sqrt p\,\frac{n}{k-1}
		=
		O_{p,\varepsilon}\!\left(n(\log n)^\varepsilon\right)
		=
		o\!\left(n^{1+\varepsilon}\right).
		\]
		Thus, for all sufficiently large \(n\),
		\[
		|A||B|
		\le
		C_{p,\varepsilon}\,n^{1+\varepsilon}p^{n/2}.
		\]
		This completes the induction for \(n\ge N_0\).  Enlarging
		\(C_{p,\varepsilon}\) to cover the finitely many dimensions \(n<N_0\) proves
		the theorem for all \(n\ge1\).
		
		Finally,
		\[
		n^{1+\varepsilon}p^{n/2}=p^{n/2+o(n)}
		\]
		for fixed \(p\), and since \(\varepsilon>0\) was arbitrary, the last assertion
		follows.
	\end{proof}

	\begin{remark}
		The theorem is stated for complete $3$-AP-free sets.  The construction
		uses complete $(2,-1)$-avoiding, indeed endpoint-midpoint complete, sets in a large
		subspace, but the padding factor in the remaining $r<4k$ coordinates is only
		a complete $3$-AP-free set.  Therefore the argument does not provide the same bound for complete $(2,-1)$-avoiding sets in every
		dimension.
	\end{remark}
	
	%\subsection{Antecedents of the method}
	
	%Several ingredients of the argument have clear predecessors in the literature. The use of completion as a saturation condition, and especially the distinction
	%between ordinary $3$-AP saturation and fixed coefficient-vector saturation, is from the earlier work of the authors \cite{CsajbokNagy2025}.  The product step above is also in that spirit: fixed-role saturation is strong enough to be transported through products, while ordinary completeness alone is not always stable under products.  %The graph-of-a-function idea belongs to a broader finite-geometric tradition in which caps, complete caps and saturating sets are constructed from algebraic varieties, especially quadrics and related configurations; see Giulietti's survey \cite{Giulietti2013}.  The subfield-flag argument follows the same general pattern, but uses a quadratic map whose surjectivity is proved by a multiplicative character sum over a subfield.  This is close in spirit to norm/trace and character-sum methods in finite geometry and coding theory.  
	The recent algebraic capset constructions of Grace
	and Voloch \cite{GraceVoloch2026} show that algebraic equations over extension
	fields can produce very small complete capsets, in particular in characteristic
	$3$. Bishnoi noted in a blog post \cite{Bishnoi2026Blog} that earlier constructions of small complete projective caps, obtained by Cossidente, Csajb\'ok, Marino and Pavese \cite{completecap}, can be used as well to construct small complete affine caps in characteristic $3$.  The present construction is different in that it works for every fixed
	odd prime $p$ and gives a complete $3$-AP-free bound in all sufficiently large
	dimensions after the padding argument, but its main idea, using
	algebraic graphs over extension fields to obtain saturation, has these earlier
	precursors.
	
	\section{Concluding remarks}
	
	We gave complete $3$-AP-free sets of size less than $2\sqrt m$ in
	every cyclic group.
	The constant $2$ in the cyclic result is still above the $3/2$ target suggested
	in the earlier paper \cite{CsajbokNagy2025}.  The present cyclic construction is deliberately
	restricted to the stronger complete $(2,-1)$-avoiding property.  For comparison, 
	complete $3$-AP-freeness allows outside points to occur either as endpoints or
	as midpoints of progressions.  Exploiting this extra freedom may lead to
	smaller constructions in cyclic groups, and is closer to the complete-cap
	viewpoint in affine spaces.
	
	The vector-space section shows that a different algebraic idea gives, for every
	fixed odd prime $p$,
	\[
	a(3\text{-}\AP,\F_p^n)
	=p^{n/2+o(n)} .
	\]
	This settles the exponent version of the corresponding problem
	for  complete $3$-AP-free sets in vector spaces over fixed odd prime
	fields.  Giving a uniform absolute constant independent of $n$ and the same bound for complete $(2,-1)$-avoiding sets in every
	dimension  remain natural open problems.
	\bigskip
	
	\noindent \textbf{AI declaration.} While the authors found  the main ideas of the proofs, they used ChatGPT during drafting and checking of parts of the exposition and for the data of the cases $m\leq 81$ in Appendix \ref{app:small-table}. They take full responsibility for the mathematical content.

	\newpage

	\appendix
	
	\section{Small-modulus certificates}\label{app:small-table}
	
	The following table gives, for each $m\le80$, an explicit complete $3$-AP-free set $A_m\subseteq\Z_m$.  It also records complete $(2,-1)$-avoiding sets $B_m\subseteq\Z_m$ when they exist.  A dash in the last column means that no complete $(2,-1)$-avoiding subset of $\Z_m$ exists.  The verification is finite: for each displayed set one checks the defining conditions directly, while the non-existence entries in the $B_m$-column follow from exhaustive search over subsets of $\Z_m$, using translation symmetry.  The construction in \cref{sec:binary} covers all $m\ge81$.

	\begin{center}
		\tiny
		\setlength{\tabcolsep}{2pt}
		\begin{longtable}{c c >{\raggedright\arraybackslash}p{0.30\textwidth} c >{\raggedright\arraybackslash}p{0.38\textwidth}}
			\toprule
			$m$ & $|A_m|$ & complete $3$-AP-free $A_m$ & $|B_m|$ & complete $(2,-1)$-avoiding $B_m$ \\
			\midrule
			\endfirsthead
			\toprule
			$m$ & $|A_m|$ & complete $3$-AP-free $A_m$ & $|B_m|$ & complete $(2,-1)$-avoiding $B_m$ \\
			\midrule
			\endhead
			1 & 1 & \(\{0\}\) & 1 & \(\{0\}\) \\
			2 & 2 & \(\{0,1\}\) & \(\infty\) & -- \\
			3 & 2 & \(\{0,1\}\) & 2 & \(\{0,1\}\) \\
			4 & 2 & \(\{0,1\}\) & 2 & \(\{0,1\}\) \\
			5 & 2 & \(\{0,1\}\) & \(\infty\) & -- \\
			6 & 4 & \(\{0,1,3,4\}\) & \(\infty\) & -- \\
			7 & 3 & \(\{0,1,3\}\) & 3 & \(\{0,1,3\}\) \\
			8 & 4 & \(\{0,1,3,4\}\) & \(\infty\) & -- \\
			9 & 4 & \(\{0,1,3,4\}\) & 4 & \(\{0,1,3,4\}\) \\
			10 & 4 & \(\{0,1,3,4\}\) & 4 & \(\{0,1,3,4\}\) \\
			11 & 4 & \(\{0,1,3,4\}\) & 4 & \(\{0,1,3,4\}\) \\
			12 & 4 & \(\{0,1,3,4\}\) & 4 & \(\{0,1,3,4\}\) \\
			13 & 4 & \(\{0,1,3,4\}\) & 4 & \(\{0,1,3,4\}\) \\
			14 & 4 & \(\{0,1,3,4\}\) & 4 & \(\{0,1,3,12\}\) \\
			15 & 4 & \(\{0,1,3,4\}\) & 4 & \(\{0,1,4,5\}\) \\
			16 & 4 & \(\{0,1,3,14\}\) & 4 & \(\{0,1,4,5\}\) \\
			17 & 4 & \(\{0,1,3,4\}\) & \(\infty\) & -- \\
			18 & 4 & \(\{0,1,4,5\}\) & \(\infty\) & -- \\
			19 & 5 & \(\{0,1,3,4,9\}\) & \(\infty\) & -- \\
			20 & 4 & \(\{0,1,5,16\}\) & \(\infty\) & -- \\
			21 & 5 & \(\{0,1,3,8,9\}\) & 5 & \(\{0,1,4,14,16\}\) \\
			22 & 6 & \(\{0,1,3,4,9,10\}\) & 6 & \(\{0,1,3,4,9,10\}\) \\
			23 & 6 & \(\{0,1,3,4,9,10\}\) & 6 & \(\{0,1,3,4,10,11\}\) \\
			24 & 6 & \(\{0,1,3,4,10,11\}\) & 6 & \(\{0,1,3,4,10,11\}\) \\
			25 & 6 & \(\{0,1,3,4,9,20\}\) & 6 & \(\{0,1,3,7,8,10\}\) \\
			26 & 6 & \(\{0,1,3,7,8,10\}\) & 6 & \(\{0,1,3,7,8,10\}\) \\
			27 & 6 & \(\{0,1,3,4,11,20\}\) & 7 & \(\{0,1,3,4,9,11,12\}\) \\
			28 & 6 & \(\{0,1,3,7,8,10\}\) & 6 & \(\{0,1,3,7,8,10\}\) \\
			29 & 6 & \(\{0,1,3,4,9,24\}\) & 7 & \(\{0,1,3,4,9,10,23\}\) \\
			30 & 6 & \(\{0,1,5,7,24,26\}\) & 7 & \(\{0,1,3,7,8,12,26\}\) \\
			31 & 6 & \(\{0,1,3,7,20,22\}\) & 7 & \(\{0,1,3,4,9,10,26\}\) \\
			32 & 6 & \(\{0,1,3,8,10,11\}\) & 8 & \(\{0,1,3,4,9,10,12,13\}\) \\
			33 & 6 & \(\{0,1,3,7,8,10\}\) & 8 & \(\{0,1,3,4,9,10,12,13\}\) \\
			34 & 7 & \(\{0,1,3,4,9,10,21\}\) & 7 & \(\{0,1,3,7,8,12,30\}\) \\
			35 & 6 & \(\{0,1,3,7,8,10\}\) & 8 & \(\{0,1,3,4,9,10,12,13\}\) \\
			36 & 7 & \(\{0,1,3,4,11,13,28\}\) & 8 & \(\{0,1,3,4,9,10,12,13\}\) \\
			37 & 7 & \(\{0,1,3,4,9,10,26\}\) & 8 & \(\{0,1,3,4,9,10,12,13\}\) \\
			38 & 7 & \(\{0,1,3,7,27,28,32\}\) & 8 & \(\{0,1,3,4,9,10,12,13\}\) \\
			39 & 6 & \(\{0,1,6,9,14,15\}\) & 8 & \(\{0,1,3,4,9,10,12,13\}\) \\
			40 & 8 & \(\{0,1,3,4,9,10,12,13\}\) & 8 & \(\{0,1,3,4,9,10,12,13\}\) \\
			41 & 7 & \(\{0,1,3,4,12,13,32\}\) & 8 & \(\{0,1,3,4,10,11,13,14\}\) \\
			42 & 7 & \(\{0,1,3,7,32,33,36\}\) & 8 & \(\{0,1,3,4,10,11,13,14\}\) \\
			43 & 7 & \(\{0,1,3,7,15,16,35\}\) & 8 & \(\{0,1,3,4,10,11,13,14\}\) \\
			44 & 7 & \(\{0,1,3,7,15,32,36\}\) & 8 & \(\{0,1,4,5,15,16,19,20\}\) \\
			45 & 8 & \(\{0,1,3,4,11,12,14,15\}\) & 8 & \(\{0,1,4,5,15,16,19,20\}\) \\
			46 & 8 & \(\{0,1,3,4,11,12,14,15\}\) & 8 & \(\{0,1,4,5,15,16,19,20\}\) \\
			47 & 8 & \(\{0,1,3,4,12,13,15,16\}\) & 8 & \(\{0,1,4,5,15,16,19,20\}\) \\
			48 & 8 & \(\{0,1,3,4,12,13,15,16\}\) & 8 & \(\{0,1,4,5,15,16,19,20\}\) \\
			49 & 8 & \(\{0,1,3,4,12,13,15,16\}\) & 8 & \(\{0,1,4,5,15,16,19,20\}\) \\
			50 & 8 & \(\{0,1,3,4,13,14,16,17\}\) & 8 & \(\{0,1,4,5,15,16,19,20\}\) \\
			51 & 8 & \(\{0,1,3,4,13,14,16,17\}\) & 8 & \(\{0,1,4,5,15,16,19,20\}\) \\
			52 & 8 & \(\{0,1,3,4,13,14,16,17\}\) & 8 & \(\{0,1,4,5,15,16,19,20\}\) \\
			53 & 8 & \(\{0,1,4,5,13,14,17,18\}\) & 8 & \(\{0,1,4,5,15,16,19,20\}\) \\
			54 & 8 & \(\{0,1,4,5,13,14,17,18\}\) & 8 & \(\{0,1,4,5,15,16,19,20\}\) \\
			55 & 8 & \(\{0,1,4,5,13,14,17,18\}\) & 8 & \(\{0,1,4,5,15,16,19,20\}\) \\
			56 & 8 & \(\{0,1,4,5,14,15,18,19\}\) & 8 & \(\{0,1,4,5,15,16,19,20\}\) \\
			57 & 8 & \(\{0,1,4,5,14,15,18,19\}\) & 8 & \(\{0,1,4,5,15,16,19,20\}\) \\
			58 & 8 & \(\{0,1,4,5,14,15,18,19\}\) & 8 & \(\{0,1,4,5,15,16,19,20\}\) \\
			59 & 8 & \(\{0,1,4,5,15,16,19,20\}\) & 8 & \(\{0,1,4,5,15,16,19,20\}\) \\
			60 & 8 & \(\{0,1,4,5,15,16,19,20\}\) & 8 & \(\{0,1,4,5,15,16,19,20\}\) \\
			61 & 8 & \(\{0,1,4,5,15,16,19,20\}\) & 8 & \(\{0,1,4,5,15,16,19,20\}\) \\
			62 & 8 & \(\{0,1,4,5,16,17,20,21\}\) & 8 & \(\{0,1,4,5,16,17,20,21\}\) \\
			63 & 8 & \(\{0,1,4,5,16,17,20,21\}\) & 8 & \(\{0,1,4,5,16,17,20,21\}\) \\
			64 & 8 & \(\{0,1,4,5,16,17,20,21\}\) & 8 & \(\{0,1,4,5,16,17,20,21\}\) \\
			65 & 8 & \(\{0,1,3,25,28,38,41,63\}\) & 10 & \(\{0,1,6,10,24,26,40,44,49,50\}\) \\
			66 & 8 & \(\{0,1,8,21,28,29,46,49\}\) & 10 & \(\{0,1,5,12,27,30,32,40,51,57\}\) \\
			67 & 8 & \(\{0,1,13,14,17,18,30,31\}\) & 10 & \(\{0,1,3,4,9,12,19,22,39,42\}\) \\
			68 & 8 & \(\{0,5,14,19,31,36,51,56\}\) & 10 & \(\{0,1,4,9,11,20,33,42,60,63\}\) \\
			69 & 9 & \(\{0,1,3,4,10,26,28,44,54\}\) & 10 & \(\{0,1,17,22,30,37,40,56,61,67\}\) \\
			70 & 8 & \(\{0,1,6,19,25,46,52,65\}\) & 10 & \(\{0,1,3,4,13,14,46,47,53,54\}\) \\
			71 & 10 & \(\{0,1,7,15,21,25,33,34,44,62\}\) & 10 & \(\{0,1,11,12,20,26,31,33,37,44\}\) \\
			72 & 8 & \(\{0,2,23,25,41,43,54,56\}\) & 10 & \(\{0,1,3,7,8,10,23,30,47,54\}\) \\
			73 & 9 & \(\{0,1,3,7,15,31,36,54,63\}\) & 9 & \(\{0,1,3,7,15,31,36,54,63\}\) \\
			74 & 10 & \(\{0,1,3,4,9,10,22,30,45,55\}\) & 11 & \(\{0,2,5,13,32,43,45,48,56,60,63\}\) \\
			75 & 10 & \(\{0,1,17,18,27,28,30,31,63,64\}\)
			& 11 & \(\{0,1,3,4,14,15,18,23,37,55,64\}\) \\
			76 & 10 & \(\{0,1,36,37,44,45,62,63,66,67\}\)
			& 10 & \(\{0,1,3,4,12,15,22,25,45,48\}\) \\
			77 & 10 & \(\{0,1,3,4,14,17,19,22,23,53\}\)
			& 11 & \(\{0,1,3,7,8,10,18,27,34,55,62\}\) \\
			78 & 10 & \(\{0,4,10,15,22,47,54,59,65,69\}\)
			& 11 & \(\{0,1,4,5,16,17,57,59,60,66,67\}\) \\
			79 & 10 & \(\{0,1,5,13,16,35,56,59,60,75\}\)
			& 12 & \(\{0,1,3,4,9,10,22,23,25,26,31,32\}\) \\
			80 & 10 & \(\{0, 1, 4, 5, 23, 24, 31, 32, 69, 70\}\)
			& 11 & \(\{0,1,3,20,22,45,48,53,57,65,78\}\) \\
			\bottomrule
		\end{longtable}
	\end{center}

	The first three columns are used only to complete the  complete $3$-AP-free bound for small $m$; the last two columns record the stronger auxiliary condition used in the $R$-complete  set construction.


\begin{thebibliography}{99}
		
		\bibitem{Behrend1946}
		F. A. Behrend,
		On sets of integers which contain no three terms in arithmetical progression,
		\emph{Proceedings of the National Academy of Sciences of the United States of America} \textbf{32} (1946), 331--332.
		
		\bibitem{Bishnoi2026Blog}
		A. Bishnoi,
		Small complete cap sets,
		\emph{Anurag's Math Blog}, March 10, 2026.
		Available at
		\url{https://anuragbishnoi.wordpress.com/2026/03/10/small-complete-cap-sets/}.
		
		\bibitem{completecap} A. Cossidente,  B. Csajb\'ok, G. Marino, F. Pavese,  Small complete caps in $\mathrm{PG}(4n+1,q)$, \emph{Bull. Lond. Math. Soc.} \textbf{55} (2023), 522--535.
		
		\bibitem{CsajbokNagy2025}
		B. Csajb\'ok and Z. L. Nagy,
		Complete 3-term arithmetic progression free sets of small size in vector spaces and other abelian groups,
		\emph{Journal of Combinatorial Theory, Series A} \textbf{215} (2025), Article 106061.
		
		\bibitem{CLP2017}
		E. Croot, V. F. Lev and P. P. Pach,
		Progression-free sets in $\Z_4^n$ are exponentially small,
		\emph{Annals of Mathematics} \textbf{185} (2017), 331--337.
		
		\bibitem{EG2017}
		J. S. Ellenberg and D. Gijswijt,
		On large subsets of $\F_q^n$ with no three-term arithmetic progression,
		\emph{Annals of Mathematics} \textbf{185} (2017), 339--343.
		
		\bibitem{Fang2021}
		J. H. Fang,
		A note on AP 3-covering sequences,
		\emph{Periodica Mathematica Hungarica} \textbf{83} (2021), 67--70.
		
		\bibitem{FGR1987}
		P. Frankl, R. L. Graham and V. R\"odl,
		On subsets of abelian groups with no 3-term arithmetic progression,
		\emph{Journal of Combinatorial Theory, Series A} \textbf{45} (1987), 157--161.
		
		\bibitem{Giulietti2013}
		M. Giulietti,
		The geometry of covering codes: small complete caps and saturating sets in Galois spaces,
		in S. R. Blackburn, S. Gerke and M. Wildon (eds.), \emph{Surveys in Combinatorics 2013}, London Mathematical Society Lecture Note Series 409, Cambridge University Press, 2013, 51--90.
		
		\bibitem{GraceVoloch2026}
		C. Grace and J. F. Voloch,
		Algebraic capsets, (2026) Journal of Combinatorial Designs, to appear. 
		
		\bibitem{HirschfeldSzonyi1991}
		J. W. P. Hirschfeld and T. Sz\H{o}nyi,
		A problem on squares in a finite field and its application to geometry,
		in \emph{Advances in Finite Geometries and Designs}, Oxford University Press, 1991, 169--176.
		
		\bibitem{KSY2018}
		S. Z. Kiss, Cs. S\'andor and Q. H. Yang,
		On generalized Stanley sequences,
		\emph{Acta Mathematica Hungarica} \textbf{154} (2018), 501--510.
		
		
		
		
		\bibitem{LidlNiederreiter1997}
		R. Lidl and H. Niederreiter,
		\emph{Finite Fields}, second edition,
		Encyclopedia of Mathematics and its Applications 20,
		Cambridge University Press, 1997.
		
		\bibitem{MartinYip2025}
		G. Martin and C. H. Yip,
		Distribution of power residues over shifted subfields and maximal cliques in generalized Paley graphs,
		\emph{Proceedings of the American Mathematical Society} \textbf{153} (2025), no. 1, 109--124.
		
		\bibitem{Meshulam1995}
		R. Meshulam,
		On subsets of finite abelian groups with no 3-term arithmetic progressions,
		\emph{Journal of Combinatorial Theory, Series A} \textbf{71} (1995), 168--172.
		
		\bibitem{Roth1953}
		K. F. Roth,
		On certain sets of integers,
		\emph{Journal of the London Mathematical Society} \textbf{28} (1953), 104--109.
		
		\bibitem{Salem}
		R. Salem,  D. C. Spencer,
		On Sets of Integers Which Contain No Three Terms in Arithmetical Progression,
		\emph{Proceedings of the National Academy of Sciences of the United States of America}
		\textbf{28} (1942), 561–563. 
		
		\bibitem{Shkredov2006}
		I. D. Shkredov,
		Szemer\'edi's theorem and problems on arithmetic progressions,
		\emph{Russian Mathematical Surveys} \textbf{61} (2006), 1101--1166.
		
	\end{thebibliography}
\end{document}